\documentclass[11pt,reqno]{article}
\usepackage{amssymb}
\usepackage{amsthm}
\usepackage{amsmath}
\usepackage{mathrsfs,verbatim,graphicx}
\usepackage{fullpage,picins,paralist}
\usepackage{hyperref}
\usepackage[latin1]{inputenc}

\newcommand\remove[1]{}

\renewcommand{\le}{\leqslant}
\renewcommand{\ge}{\geqslant}
\renewcommand{\leq}{\leqslant}
\renewcommand{\geq}{\geqslant}

\newcommand{\1}{\mathbf{1}}

\newcommand{\e}{\varepsilon}
\newcommand{\R}{\mathbb{R}}

\newcommand{\E}{\mathbb{E}}
\newcommand{\U}{\mathcal{U}}

\newcommand{\N}{\mathbb{N}}

\newcommand{\C}{\mathbb{C}}
\newcommand{\f}{\varphi}

\newtheorem{theorem}{Theorem}[section]

\newtheorem{lemma}[theorem]{Lemma}

\newtheorem{claim}[theorem]{Claim}

\newtheorem{corollary}[theorem]{Corollary}

\theoremstyle{theorem}

\newcommand{\eqdef}{\stackrel{\mathrm{def}}{=}}

\renewcommand{\H}{\mathbb{H}}

\renewcommand{\setminus}{\smallsetminus}

\begin{document}

\title{Sharp quantitative nonembeddability of the Heisenberg group\\ into superreflexive Banach spaces}

\author{
Tim Austin\thanks{Partially supported by a fellowship from Microsoft Corporation.}\\Brown
\and
Assaf Naor\thanks{Supported by NSF grants CCF-0635078 and CCF-0832795, BSF
grant 2006009, and the Packard Foundation.}\\ NYU
\and Romain Tessera\\E.N.S. Lyon
}
\date{}
\maketitle

\begin{abstract}
Let $\H$ denote the discrete Heisenberg group, equipped with a word metric $d_W$ associated to some finite symmetric generating set. We show that if $(X,\|\cdot\|)$ is a $p$-convex Banach space then for any Lipschitz function $f:\H\to X$ there exist $x,y\in \H$ with $d_W(x,y)$ arbitrarily large and
\begin{equation}\label{eq:comp abs}
\frac{\|f(x)-f(y)\|}{d_W(x,y)}\lesssim \left(\frac{\log\log d_W(x,y)}{\log d_W(x,y)}\right)^{1/p}.
\end{equation}
 We also show that any embedding into $X$ of a ball of radius $R\ge 4$ in $\H$ incurs bi-Lipschitz distortion that grows at least as a constant multiple of
\begin{equation}\label{eq:dist abs}
\left(\frac{\log R}{\log\log R}\right)^{1/p}.
\end{equation}
 Both~\eqref{eq:comp abs} and~\eqref{eq:dist abs} are sharp up to the iterated logarithm terms. When $X$ is
Hilbert space we obtain a representation-theoretic proof yielding
bounds  corresponding to~\eqref{eq:comp abs} and~\eqref{eq:dist abs}
which are sharp up to a universal constant.
\end{abstract}


\section{Introduction}

Let $\H\eqdef \left.\left\langle a,b\right| aba^{-1}b^{-1}\ \mathrm{is\ central}\right\rangle$ denote the discrete Heisenberg group, with canonical generators $a,b\in \H$. We let $d_W(\cdot,\cdot)$ denote the left-invariant word metric on $\H$ associated to the symmetric generating set $S\eqdef \{a,b,a^{-1},b^{-1}\}$.

A Banach space $(X,\|\cdot\|_X)$ is {\em superreflexive} if it admits an equivalent uniformly convex norm, i.e., a norm $\|\cdot\|$ satisfying $\alpha\|x\|_X\le \|x\|\le\beta\|x\|_X$ for some $\alpha,\beta>0$ and all $x\in X$, such that for all $\e\in (0,1)$ there exists $\delta>0$ for which we have
\begin{equation}\label{eq:def UC}
\|x\|=\|y\|=1\ \wedge\ \|x-y\|=\e\implies \|x+y\|\le 2-\delta.
\end{equation}

Here we prove the following result:
\begin{theorem}\label{thm:some power}
Let $(X,\|\cdot\|_X)$ be a superreflexive Banach space. Then there exist $c,C>0$ such that for every $f:\H\to X$ which is $1$-Lipschitz with respect to the metric $d_W$, there are $x,y\in X$ with $d_W(x,y)$ arbitrarily large and
$$
\frac{\|f(x)-f(y)\|_X}{d_W(x,y)}\le \frac{C}{\left(\log d_W(x,y)\right)^c}\enspace .
$$
\end{theorem}
The fact that $\H$ does not admit a bi-Lipschitz embedding into any superreflexive Banach space was proved in~\cite{LN06,GFDA}. These proofs use  an argument of Semmes~\cite{Sem96}, based on a natural extension of Pansu's differentiability theorem~\cite{Pan89}.

A natural way to quantify the extent to which $\H$ does not admit a
bi-Lipschitz embedding into $(X,\|\cdot\|_X)$ is via Gromov's
notion~\cite[Sec. 7.3]{Gro93} of {\em compression rate}, defined for
a Lipschitz function $f:\H\to X$ as the largest
function $\omega_f:(0,\infty)\to [0,\infty)$ such that for all
$x,y\in \H$ we have $\|f(x)-f(y)\|_X\ge
\omega_f\left(d_W(x,y)\right)$. The fact that $\H$ does not admit a
bi-Lipschitz embedding into a superreflexive Banach space $X$ means
that $\liminf_{t\to\infty} \omega_f(t)/t=0$ for all Lipschitz
functions $f:\H\to X$. The differentiability-based proof of this
nonembeddability result involves a limiting argument that does not
give information on the rate at which $\omega_f(t)/t$ vanishes.
Theorem~\ref{thm:some power} supplies such information, via an
approach which is different from the arguments in~\cite{LN06,GFDA}.


Cheeger and Kleiner proved~\cite{ckbv} that $\H$ does not admit a
bi-Lipschitz embedding into $L_1$. In~\cite{ckn} it was shown that
there exists $c>0$ such that for any Lipschitz function $f:\H\to
L_1$ we have $\omega_f(t)/t\le 1/(\log t)^c$ for arbitrarily large
$t$. This result covers Theorem~\ref{thm:some power} when the
superreflexive Banach space $X$ admits a bi-Lipschitz embedding into
$L_1$: such spaces include $L_p$ for $p\in (1,2]$.
Theorem~\ref{thm:some power} is new even for spaces such as $L_p$
for $p \in (2,\infty)$, which do not admit a bi-Lipschitz embedding into
$L_1$ (see~\cite{BL}). Moreover, our method yields sharp results,
while the constant $c$ obtained in~\cite{ckn} is far from sharp.

In order to state our sharp version of Theorem~\ref{thm:some power},
we recall the following important theorem of Pisier~\cite[Thm.
3.1]{Pisier-martingales}: if $X$ is superreflexive then it admits an
equivalent norm $\|\cdot\|$ for which there exist $p\ge 2$ and
$K>0$ satisfying the following improvement of~\eqref{eq:def UC}:
\begin{equation}\label{eq:p-convexity}
\forall x,y\in X,\quad \left\|\frac{x+y}{2}\right\|^p\le \frac{\|x\|^p+\|y\|^p}{2}-\frac{1}{K^p}\left\|\frac{x-y}{2}\right\|^p.
\end{equation}
A Banach space admitting an equivalent norm
satisfying~\eqref{eq:p-convexity} is said to be $p$-convex. If
$(X,\|\cdot\|)$ satisfies~\eqref{eq:p-convexity} then the infimum
over those $K>0$ satisfying~\eqref{eq:p-convexity} is denoted
$K_p(X)$. For concreteness, when $p\in (1,2]$ we have $K_2(L_p)\le 1/\sqrt{p-1}$  and for $p\ge 2$ we have $K_p(L_p)\le 1$ (see~\cite{BCL}).

The following theorem is a refinement of Theorem~\ref{thm:some
power}.

\begin{theorem}\label{thm:main} Assume that the Banach space $(X,\|\cdot\|)$ satisfies~\eqref{eq:p-convexity}.
Let $f:\H\to X$ be a $1$-Lipschitz function. Then for every $t\ge 3$
there exists an integer $t\le n\le t^2$ such
that\footnote{In~\eqref{eq:asymp}, and in the rest of this paper,
the notation $\lesssim,\gtrsim$ denotes the corresponding
inequalities up to a universal multiplicative factor. The notation
$A\asymp B$ stands for $A\lesssim B\wedge B\lesssim A$.}
\begin{equation}\label{eq:asymp}
\frac{\omega_f(n)}{n}\lesssim K_p(X)\left(\frac{\log\log n}{\log
n}\right)^{1/p}.
\end{equation}
\end{theorem}
The estimate~\eqref{eq:asymp} is sharp up to the iterated logarithm
term. Indeed, $L_p$ is $p$-convex when $p\in [2,\infty)$, and
in~\cite{Tess,Tess08} it was shown that there exists $f:\H\to L_p$
satisfying
$$
\frac{\omega_f(n)}{n}\gtrsim \frac{1}{(\log n)^{1/p}\log\log n}
$$
for all $n\ge 3$ (we refer to~\cite{Tess08} for a more refined
result of this type).

Our proof of Theorem~\ref{thm:main} circumvents the difficulties
involved with proving quantitative variants of differentiability
results by avoiding the need to reason about arbitrary Lipschitz
mappings. Instead, we start by using a simple result
from~\cite{NP08} which reduces the problem to equivariant mappings.
Specifically, in~\cite[Thm. 9.1]{NP08} it is shown that if $X$
satisfies~\eqref{eq:p-convexity} and $f:\H\to X$ is $1$-Lipschitz,
then there exists a Banach space $Y$ that also
satisfies~\eqref{eq:p-convexity}, with $K_p(Y)=K_p(X)$ (in fact, $Y$
is finitely representable in $\ell_p(X)$), an action $\pi$ of $\H$
on $Y$ by linear isometric automorphisms, and a $1$-cocycle $F:\H\to
Y$ (i.e., $F(xy)=\pi(x)F(y)+F(x)$ for all $x,y\in \H$) with
$\omega_F=\omega_f$. Thus, in proving Theorem~\ref{thm:main} it
suffices to assume that $f$ itself is a $1$-cocycle. We note that if
$X$ is Hilbert space then $Y$ is also Hilbert space; this is an
older result of Gromov (see~\cite{CTV07}). More generally, when
$X=L_p$ then it is shown in~\cite{NP08} that we can take $Y=L_p$.

Having reduced the problem to $1$-cocycles, our starting point is a
(non-quantitative) proof, explained in Section~\ref{sec:non quant},
showing that if $X$ is an ergodic Banach space, then for every
$1$-cocycle $f:\H\to X$ we have $\liminf_{t\to\infty}
\omega_f(t)/t=0$. It turns out that the ideas of this proof, which
crucially use the fact that $f$ is a $1$-cocycle, can be
(nontrivially) adapted to yield Theorem~\ref{thm:main}.

Recall that $X$ is ergodic if for every linear isometry $T:X\to X$ and every
$x\in X$ the sequence $\left\{\frac{1}{n}\sum_{j=0}^{n-1}
T^jx\right\}_{n=1}^\infty$ converges in norm. Reflexive spaces, and
hence also  superreflexive spaces, are ergodic (see~\cite[p.
662]{DS58}). If a Banach space $X$ has the property
that all Banach spaces that are finitely representable in $X$ are
ergodic, then $X$ must be superreflexive~\cite{BS72}. Thus, when
using the reduction to $1$-cocyles based on the result
of~\cite{NP08}, the class of Banach spaces to which it naturally
applies is the class of superreflexive spaces.

\subsection{Bi-Lipschitz distortion of balls}\label{sec:balls}

For $R\ge 1$ let $B_R=\{x\in \H:\ d_W(e,x)\le R\}$ denote the ball
of radius $R$ centered at the identity element $e\in \H$. The
bi-Lipschitz distortion of $(B_R,d_W)$ in $(X,\|\cdot\|)$, denoted
$c_X(B_R)$, is the infimum over those $D\ge 1$ such that there
exists $f:B_R\to X$ satisfying \begin{equation}\label{eq:def
distortion} \forall\  x,y\in B_R,\quad d_W(x,y)\le \|f(x)-f(y)\|\le
Dd_W(x,y).
\end{equation}

 Another way to measure the extent to which $\H$ does not admit a
bi-Lipschitz embedding into $X$ is via the rate at which $c_X(B_R)$ grows to
$\infty$ with $R$. In~\cite{ckn} it was shown that
\begin{equation}\label{eq:ckn}
c_{L_1}(B_R)\gtrsim (\log R)^c
\end{equation}
for some universal constant $c>0$. This result is of  importance due to an application to theoretical computer science; see~\cite{CKN09,ckn,Naor10} for a detailed discussion. Evaluating the supremum over those $c>0$ satisfying~\eqref{eq:ckn} remains an important open problem. Theorem~\ref{thm:main} implies the following sharp bound on the bi-Lipschitz distortion of $B_R$ into a $p$-convex Banach space:

\begin{theorem}\label{thm:distortion}
If a Banach space $(X,\|\cdot\|)$ satisfies~\eqref{eq:p-convexity} then for every $R\ge 4$ we have
$$
c_X(B_R) \gtrsim \frac{1}{K_p(X)}\left(\frac{\log R}{\log\log R}\right)^{1/p}.
$$
\end{theorem}
Thus in particular for $p\in (1,2]$ we have $c_{L_p}(B_R) \gtrsim \sqrt{p-1}\cdot  (\log R)^{\frac12-o(1)}$ and  for $p\ge 2$ we have  $c_{L_p}(B_R) \gtrsim (\log R)^{\frac{1}{p}-o(1)}$. Theorem~\ref{thm:distortion} is a formal consequence of Theorem~\ref{thm:main}. The  simple deduction of Theorem~\ref{thm:distortion} from Theorem~\ref{thm:main} is presented in Section~\ref{sec:deduce}. It follows from the results of~\cite{Ass83, Rao99} (see the
explanation in~\cite{GKL03,Tess08}) that for every $p\ge 2$ we have $c_{L_p}(B_R)\lesssim (\log R)^{1/p}$. Thus Theorem~\ref{thm:distortion} is sharp up to iterated logarithms.

\subsection{The case of Hilbert space}

In section~\ref{sec:hilbert} we prove the following Poincar\'e-type
inequality for functions on $\H$ taking value in Hilbert space:

\begin{theorem}\label{thm:ingtro discrete}
For every $f:\H\to L_2$ and every $R\in \N$ we have
\begin{equation}\label{eq:poin intro}
\sum_{x\in
B_R}\sum_{k=1}^{R^2}\frac{\left\|f(xc^k)-f(x)\right\|_2^2}{k^2}\lesssim
\sum_{x\in B_{22R}}
\left(\|f(xa)-f(x)\|^2_2+\|f(xb)-f(x)\|^2_2\right).
\end{equation}
\end{theorem}
This result has the following two sharp consequences. First, assume
that $\theta:(0,\infty)\to [0,\infty)$ is nondecreasing, and that
$\theta\le \omega_f$ for some $1$-Lipschitz $f:\H\to L_2$. Then
since $|B_{22R}|\asymp |B_R|$ and $d_W(c^k,e)\asymp \sqrt{k}$ for all
$k\in \N$, inequality~\eqref{eq:poin intro} implies that
\begin{equation}\label{eq:sharp integral}
\frac{1}{2}\int_{1}^\infty \frac{\theta\left(t\right)^2}{t^3}dt =\int_{1}^\infty
\frac{\theta\left(\sqrt{s}\right)^2}{s^2}ds \lesssim
\sum_{k=1}^\infty \frac{\theta\left(\sqrt{k}\right)^2}{k^2}\lesssim
1.
\end{equation}
 Combined with \cite[Theorem 1]{Tess}, we obtain
\begin{corollary}
A nondecreasing function $\theta:(0,\infty)\to [0,\infty)$ satisfies
$\theta \le \omega_{f}$ for some Lipschitz function $f: \H\to L_2$
if and only if
 \begin{equation}\int_1^{\infty}\left(\frac{\theta(t)}{t}\right)^2 \frac{dt}{t}<\infty.
 \end{equation}\label{compressionChar}
\end{corollary}

A second corollary of Theorem~\ref{thm:ingtro discrete} yields a
sharp bound (up to universal constants) on $c_{L_2}(B_R)$. Indeed,
fix $R\ge 2$ and assume that $f:B_{R}\to L_2$ satisfies $d_W(x,y)\le
\|f(x)-f(y)\|\le Dd_W(x,y)$ for all $x,y\in B_{R}$. Let $f^*:\H\to
L_2$ have Lipschitz constant at most $2D$ and coincide with $f$ on
$B_{R/2}$ (see equation~\eqref{eq:extend} for an explicit formula
defining such an extension $f^*$). It follows from~\eqref{eq:poin
intro} applied to $f^*$ that $D^2\gtrsim \sum_{k=1}^R
\frac{1}{k}\gtrsim \log R$. Thus $c_{L_2}(B_R)\gtrsim \sqrt{\log
R}$. In conjunction with the previously quoted upper bound on
$c_{L_2}(B_R)$, we have
\begin{corollary}\label{eq:hilbert distortion}
For every $R\ge 2$ we have $c_{L_2}(B_R)\asymp \sqrt{\log R}$.
\end{corollary}

Roughly speaking, the proof of Theorem~\ref{thm:ingtro discrete}
proceeds via a reduction to the case of $1$-cocycles corresponding
to the irreducible representations of $\H$ (see Section
\ref{sec:hilbert}). But actually, since the representation theory of
the continuous Heisenberg group is simpler than the representation
theory of the discrete Heisenberg group $\H$, we  first apply a
discretization argument which reduces Theorem~\ref{thm:ingtro
discrete} to an inequality on the real Heisenberg group. Then an
averaging argument reduces the proof to an inequality on cocycles.
Every unitary representation of the continuous Heisenberg group decomposes as a direct
integral of irreducibles, and cocycles themselves can be
decomposed accordingly. Since the desired inequality involves a sum of squares of norms, it suffices to prove it for cocycles corresponding to irreducible
representations (that is, for each direct integrand separately). The computation for irreducible representations is
carried out in Section~\ref{sec:irreducible}.

\section{Sublinear growth of Heisenberg cocycles in ergodic
spaces}\label{sec:non quant}

Write $c\eqdef [a,b]=aba^{-1}b^{-1}$. Thus $c$ lies in the center of $\H$ and for every $n\in \N$ we have $d_W\left(c^{n^2},e_\H\right)=4n$ (in fact $c^{n^2}=[a^n,b^n]=a^nb^na^{-n}b^{-n}$).

Let $(X,\|\cdot \|)$ be a Banach space  and $\pi:G\to \mathrm{Aut}(X)$ be an action of $\H$ on $X$ by linear isometric automorphisms.  In addition let $f\in Z^1(\pi)$ be  a $1$-cocycle, so $f:\H\to X$ and for all $x,y\in\H$ we have $f(xy)=\pi(x)f(y)+f(x)$. We assume in what follows that $f$ is $1$-Lipschitz, or equivalently that $\max\{\|f(a)\|,\|f(b)\|\}=1$.

In this section we quickly show that if $X$ is an ergodic Banach space then $\liminf_{t\to\infty}\omega_f(t)/t = 0$, but without obtaining any quantitative bounds.

If $X$ is ergodic then the operator on $X$ defined by
\[Px \eqdef \lim_{N\to\infty}\frac{1}{N}\sum_{n=0}^{N-1}\pi(c)^{n}x\]
is a contraction onto the subspace $X_0 \subseteq X$ of $\pi(c)$-invariant vectors, and since $Px = x$ for any $x \in X_0$ it follows at once that $P$ is idempotent.  Also, since $c$ is central in $\H$, the projection $P$ commutes with $\pi(g)$ for all $g\in \H$, and hence the maps $g\mapsto Pf(g)$ and $g\mapsto (I - P)f(g)$ are both still members of $Z^1(\pi)$.  Since $P$ and $I - P$ are bounded, these cocycles are both still Lipschitz functions from $\H$ to $X$.

We complete the proof by showing that for any $\varepsilon > 0$ we have
\begin{equation}\label{eq:eps}
\left\|f\left(c^{N^2}\right)\right\| \leq 2\varepsilon d_W\left(e_\H,c^{N^2}\right) \le 8\varepsilon N
 \end{equation}
 for all sufficiently large $N$. To prove this we consider the two cocycles $Pf$ and $(I-P)f$ separately.  On the one hand, $Pf$ takes values among the $\pi(c)$-invariant vectors, and hence the cocycle identity implies that
\[Pf\left(c^N\right) = \sum_{n=0}^{N-1}\pi(c)^n Pf(c) = NPf(c),\]
and therefore $\left\|Pf\left(c^N\right)\right\| = N\|Pf(c)\|$.  However, $\left\|Pf\left(c^N\right)\right\|\le d_W\left(e_\H,c^N\right)\lesssim \sqrt{N}$, so these relations are compatible only if $Pf(c) = 0$.

On the other hand, let $\widetilde{f} \eqdef (I-P)f$ and for each $K\geq 1$,
\[v_K \eqdef -\frac{1}{K}\sum_{k=1}^K f\left(c^k\right).\]
Observe from the cocycle identity and the centrality of $c$ that
\begin{multline}\label{eq:v_K invariant}
-\pi(g)v_K + \widetilde{f}(g) = \frac{1}{K}\sum_{k=1}^K \left(\pi(g)\widetilde{f}\left(c^k\right) + \widetilde{f}(g)\right) = \frac{1}{K}\sum_{k=1}^K \widetilde{f}\left(gc^k\right)\\ = \frac{1}{K}\sum_{k=1}^K \widetilde{f}\left(c^kg\right) = \frac{1}{K}\sum_{k=1}^K \pi\left(c^k\right)\widetilde{f}(g) - v_K.
\end{multline}

Re-arranging~\eqref{eq:v_K invariant} gives
\begin{equation}\label{eq:tilde}
\widetilde{f}(g) = \pi(g)v_K - v_K + \frac{1}{K}\sum_{k=1}^K\pi(c)^k\widetilde{f}(g).\end{equation}
For any fixed $g \in \H$ the last term of this right-hand side of~\eqref{eq:tilde} converges to $P(I-P)f(g) = 0$ (using again that $X$ is ergodic), and so in particular once $K$ is sufficiently large we obtain that for all $g\in \H$ we have,
\[\max\left\{\left\|\widetilde{f}\left(a^{\pm 1}\right) - \left(\pi\left(a^{\pm 1}\right)v_K - v_K\right)\right\|,\ \left\|\widetilde{f}\left(b^{\pm 1}\right) - \left(\pi\left(b^{\pm 1}\right)v_K - v_K\right)\right\|\right\} \leq \varepsilon.\]

Having obtained this approximation to $\widetilde{f}$ by a coboundary, let $c^{N^2} = s_1s_2\cdots s_{4N}$ be an expression for $c^{N^2}$ as a word in $S$, and observe from another appeal to the cocycle identity that
\begin{multline*}
\widetilde{f}\left(c^{N^2}\right) = \sum_{i=0}^{4N - 1}\pi\left(s_1s_2\cdots s_i\right)\widetilde{f}\left(s_{i+1}\right) = \sum_{i=0}^{4N - 1}\pi\left(s_1s_2\cdots s_i\right)\left(\pi\left(s_{i+1}\right)v_K - v_K\right) + R_N\\
= \pi(s_1s_2\cdots s_{4N})v_K - v_K + R_N
\end{multline*}
for some remainder $R_N$ which is a sum of $4N$ terms all of norm at most $\varepsilon$.  Since the action $\pi$ is isometric and we may let $N$ grow independently of $K$ we obtain
\[\left\|\widetilde{f}\left(c^{N^2}\right)\right\| \leq 2\|v_K\| + \|R_N\| \leq 8\varepsilon N\]
for all sufficiently large $N$.  Since $\varepsilon$ was arbitrary and $\widetilde{f}\left(c^{N^2}\right) = f\left(c^{N^2}\right)$ by our analysis of $Pf$ above, this completes the proof of~\eqref{eq:eps}.

\section{A uniform convexity lemma for ergodic averages}\label{sec:erg}

We prove here a simple lemma on the behavior of ergodic averages in $p$-convex Banach spaces.

\begin{lemma}\label{lem:cauchy}
Assume that $(X,\|\cdot\|)$ satisfies~\eqref{eq:p-convexity}. Fix $z\in X$ and an operator $T:X\to X$ with $\|T\|\le 1$. For every integer $n\ge 0$ denote
$$
s_n\eqdef \frac{1}{2^n}\sum_{j=0}^{2^n-1} T^jz.
$$
Then for every $\ell\in \N$ we have:
\begin{equation}\label{eq:summed version}
\sum_{i=0}^\infty \frac{1}{2^\ell}\sum_{j=0}^{2^\ell-1}\left\|s_{(i+1)\ell}-T^{j2^{i\ell}}s_{i\ell}\right\|^p\le (2K)^p\|z\|^p.
\end{equation}
\end{lemma}

\begin{proof} A consequence of~\eqref{eq:p-convexity} is that for every $x_1,\ldots,x_n\in X$ we have:
\begin{equation}\label{eq:convex averages}
\frac{1}{n}\sum_{i=1}^n\left\|x_i-\frac{1}{n}\sum_{j=1}^n x_j\right\|^p\le (2K)^p\left(\frac{1}{n}\sum_{i=1}^n\|x_i\|^p-\left\|\frac{1}{n}\sum_{i=1}^nx_i\right\|^p\right).
\end{equation}
For the derivation of~\eqref{eq:convex averages} from~\eqref{eq:p-convexity} see~\cite[Lem. 3.1]{MN10}.

Due to the identity
$$
s_{(i+1)\ell}=\frac{1}{2^\ell}\sum_{j=0}^{2^\ell-1}T^{j2^{i\ell}}s_{i\ell},
$$
inequality~\eqref{eq:convex averages} implies that:
\begin{multline}\label{eq:for telescope}
\frac{1}{2^\ell}\sum_{j=0}^{2^\ell-1} \left\|T^{j2^{i\ell}}s_{i\ell}-s_{(i+1)\ell}\right\|^p\le (2K)^p\left(\frac{1}{2^\ell}\sum_{j=0}^{2^\ell-1}
\left\|T^{j2^{i\ell}}s_{i\ell}\right\|^p-\left\|s_{(i+1)\ell}\right\|^p\right)\\\le (2K)^p\left(\left\|s_{i\ell}\right\|^p-\left\|s_{(i+1)\ell}\right\|^p\right).
\end{multline}
The desired inequality~\eqref{eq:summed version} now follows by summing~\eqref{eq:for telescope} over $i\in\{0,1,\ldots\}$.
\end{proof}

\section{Estimates for Heisenberg cocycles}\label{sec:cocycles}

Let $\pi:\H\to\mathrm{Aut}(X)$ and $f \in Z^1(\pi)$ be as in Section~\ref{sec:non quant}.  For every $n\in \N$ define a linear operator $P_n:X\to X$ by
$$
P_n\eqdef  \frac{1}{2^n}\sum_{j=0}^{2^n-1} \pi(c)^j.
$$
Thus $\|P_n\|\le 1$.

\begin{lemma}\label{lem:ab}
Assume that $(X,\|\cdot\|)$ satisfies~\eqref{eq:p-convexity}. Then for every $\ell,k,m\in \N$ there exist integers $i\in [k+1,k+m]$ and $j\in [0,2^\ell-1]$ satisfying for all $n\in \N$,
\begin{equation}\label{eq:n^2}
\left\|\pi\left(c^{-j2^{i\ell}}\right)P_{(i+1)\ell}
f\left(c^{n^2}\right)-P_{i\ell}f\left(c^{n^2}\right)\right\|\le \frac{16Kn}{m^{1/p}}.
\end{equation}
\end{lemma}

\begin{proof} Consider the Banach space $Y=X\oplus X$, equipped with the norm $$\|(x,y)\|_Y=\left(\|x\|^p+\|y\|^p\right)^{1/p}.$$
We also define $T:Y\to Y$ by $T(x,y)=(\pi(c)x,\pi(c)y)$. Then $\|T\|\le 1$. Since $(Y,\|\cdot\|_Y)$ satisfies~\eqref{eq:p-convexity} we may apply Lemma~\ref{lem:cauchy} to $z=(f(a),f(b))\in Y$, obtaining the following estimate:
\begin{eqnarray*}\label{eq:telescoped}
&&\!\!\!\!\!\!\!\!\!\!\!\!\!\!(4K)^p\ge\sum_{i=k+1}^{k+m}\frac{1}{2^\ell}\sum_{j=0}^{2^\ell-1}\left(\left\|P_{(i+1)\ell}
f\left(a\right)-\pi\left(c^{j2^{i\ell}}\right)P_{i\ell}f\left(a\right)\right\|^p+\left\|P_{(i+1)\ell}
f\left(b\right)-\pi\left(c^{j2^{i\ell}}\right)P_{i\ell}f\left(b\right)\right\|^p\right)\\
&=&\sum_{i=k+1}^{k+m}\frac{1}{2^\ell}\sum_{j=0}^{2^\ell-1}\left(\left\|\pi\left(c^{-j2^{i\ell}}\right)P_{(i+1)\ell}
f\left(a\right)-P_{i\ell}f\left(a\right)\right\|^p+\left\|\pi\left(c^{-j2^{i\ell}}\right)P_{(i+1)\ell}
f\left(b\right)-P_{i\ell}f\left(b\right)\right\|^p\right)\\
&\ge& m\min_{\substack{k+1\le i\le k+m\\0\le j\le 2^\ell-1}}\left(\left\|\pi\left(c^{-j2^{i\ell}}\right)P_{(i+1)\ell}
f\left(a\right)-P_{i\ell}f\left(a\right)\right\|^p+\left\|\pi\left(c^{-j2^{i\ell}}\right)P_{(i+1)\ell}
f\left(b\right)-P_{i\ell}f\left(b\right)\right\|^p\right).
\end{eqnarray*}
It follows that there exist integers $i\in [k+1,k+m]$, $j\in [0,2^\ell-1]$ such that
\begin{equation}\label{eq; use lemma square}
\max\left\{\left\|\pi\left(c^{-j2^{i\ell}}\right)P_{(i+1)\ell}
f\left(a\right)-P_{i\ell}f\left(a\right)\right\|,\left\|\pi\left(c^{-j2^{i\ell}}\right)P_{(i+1)\ell}
f\left(b\right)-P_{i\ell}f\left(b\right)\right\|\right\}\le\frac{4K}{m^{1/p}}.
\end{equation}

Consider the operator
$$
Q_n\eqdef \frac{1}{n}\sum_{i=0}^{n-1} \pi(a)^{i}.
$$
The cocycle identity implies that $f(a^n)=nQ_nf(a)$. Thus
\begin{multline}\label{eq:commute}
\pi\left(c^{-j2^{i\ell}}\right)P_{(i+1)\ell}f(a^n)-P_{i\ell}f(a^n)=n\left(\pi\left(c^{-j2^{i\ell}}\right)
P_{(i+1)\ell}-P_{i\ell}\right)Q_nf(a)\\
=nQ_n\left(\pi\left(c^{-j2^{i\ell}}\right)
P_{(i+1)\ell}-P_{i\ell}\right)f(a),
\end{multline}
where the last equality in~\eqref{eq:commute} holds since $c$ is in the center of $\H$, and therefore $Q_n$ commutes with all of $\{P_r\}_{r=0}^\infty$. Since $\|Q_n\|\le 1$, it follows from~\eqref{eq:commute} and~\eqref{eq; use lemma square} that
\begin{equation}\label{eq:a^n}
\left\|\pi\left(c^{-j2^{i\ell}}\right)P_{(i+1)\ell}f(a^n)-P_{i\ell}f(a^n)\right\|\le \frac{4Kn}{m^{1/p}}.
\end{equation}
Since $f(a^{-n})=-\pi(a)^{-n}f(a^n)$, and $\pi(a)$ commutes with $\pi(c)$, and hence  with all of $\{P_r\}_{r=0}^\infty$, it follows that also
\begin{equation}\label{eq:a^{-n}}
\left\|\pi\left(c^{-j2^{i\ell}}\right)P_{(i+1)\ell}f(a^{-n})-P_{i\ell}f(a^{-n})\right\|\le \frac{4Kn}{m^{1/p}}.
\end{equation}
An identical argument implies the analogous bounds with $a$ replaced by $b$:
\begin{equation}\label{eq:b^n}
\left\|\pi\left(c^{-j2^{i\ell}}\right)P_{(i+1)\ell}f(b^n)-P_{i\ell}f(b^n)\right\|\le \frac{4Kn}{m^{1/p}}
\end{equation}
and
\begin{equation}\label{eq:b^{-n}}
\left\|\pi\left(c^{-j2^{i\ell}}\right)P_{(i+1)\ell}f(b^{-n})-P_{i\ell}f(b^{-n})\right\|\le \frac{4Kn}{m^{1/p}}.
\end{equation}

The cocycle identity implies that for all $n\in \N$,
\begin{equation*}\label{eq:cocycle ab}
f\left(c^{n^2}\right)=f\left([a^n,b^n]\right)=\pi(a^nb^na^{-n})f(b^{-n})+\pi(a^nb^n)f(a^{-n})+\pi(a^n)f(b^n)+f(a^n).
\end{equation*}
Thus, using~\eqref{eq:a^n}, \eqref{eq:a^{-n}}, \eqref{eq:b^n} and \eqref{eq:b^{-n}}, we conclude the validity of~\eqref{eq:n^2}.
\end{proof}

\begin{lemma}\label{lem: P_m bound}
For every $m,n\in \N$ we have
\begin{equation}\label{eq:P_m small}
\left\|P_mf\left(c^{n^2}\right)\right\|\lesssim \frac{n^{5/3}}{2^{m/3}}.
\end{equation}
\end{lemma}

\begin{proof}
Note that for every $k\in \N$ we have
$$
P_m-\pi\left(c^{k}\right)P_m=\frac{1}{2^m}\sum_{j=0}^{k-1}\pi(c)^j-\frac{1}{2^m}\sum_{j=2^m}^{2^m+k-1}\pi(c)^j.
$$
Thus,
\begin{equation}\label{eq:folner}
\left\|P_m-\pi\left(c^{k}\right)P_m\right\|\le \frac{2k}{2^m}.
\end{equation}
The cocycle identity implies that
\begin{equation}\label{eq:identity k}
f\left(c^{(kn)^2}\right)=\sum_{j=0}^{k^2-1}\pi\left(c^{jn^2}\right)f\left(c^{n^2}\right).
\end{equation}
Using the fact that $f$ is $1$-Lipschitz, $\|P_m\|\le 1$ and $d_W\left(e_\H,c^{(kn)^2}\right)\le 4kn$, we deduce from~\eqref{eq:identity k} that
\begin{multline*}
4kn\ge \left\|P_mf\left(c^{(kn)^2}\right)\right\|\ge \sum_{j=0}^{k^2-1} \left(\left\|P_m f\left(c^{n^2}\right)\right\|-\left\|P_m-\pi\left(c^{jn^2}\right)P_m\right\|\cdot \left\| f\left(c^{n^2}\right)\right\|\right)\\\stackrel{\eqref{eq:folner}}{\ge} k^2\left\|P_m f\left(c^{n^2}\right)\right\|-\sum_{j=0}^{k^2-1} \frac{2jn^2}{2^m}\cdot 4n\ge k^2\left\|P_m f\left(c^{n^2}\right)\right\|-\frac{4n^3k^4}{2^m}.
\end{multline*}
Thus,
\begin{equation}\label{eq:optimize}
\left\|P_m f\left(c^{n^2}\right)\right\|\le \frac{4n}{k}+\frac{4n^3k^2}{2^m}.
\end{equation}
Choosing $k=\left\lceil (2^{m-1}/n^2)^{1/3}\right\rceil$ in~\eqref{eq:optimize} (roughly the optimal choice of $k$), we obtain~\eqref{eq:P_m small}.
\end{proof}

\begin{lemma}\label{lem:I-P} For every $m,n\in \N$ we have:
\begin{equation}\label{eq:I-P}
\left\|f\left(c^{n^2}\right)-P_mf\left(c^{n^2}\right)\right\|\le 2^{m/3}n^{1/3}.
\end{equation}
\end{lemma}

\begin{proof} In this proof the relation to Section~\ref{sec:non quant} becomes clear. Define $\widetilde f: \H\to X$ by
$$
\widetilde f(h)\eqdef f\left(h\right)-P_mf\left(h\right)=(I-P_m)f(h).
$$
Note that $\widetilde f\in Z^1(\pi)$. Fix an integer $k\ge 1$ that will be determined later.
Consider the vector $v\in X$ defined by
\begin{equation*}\label{eq:def v}
v\eqdef -\frac{1}{k} \sum_{j=0}^{k-1} \widetilde f\left(c^{j}\right).
\end{equation*}
Then
\begin{equation}\label{eq:norm v}
\|v\|\lesssim \frac{1}{k}\sum_{j=0}^{k-1} \sqrt{j}\lesssim \sqrt{k}.
\end{equation}

Since  $c$ is in the center of $\H$, we have the following identity for every $h\in \H$:
\begin{multline}\label{eq:get coboundary}
-\pi(h)v+\widetilde f(h)=\frac{1}{k}\sum_{j=0}^{k-1}\left(\pi(h)\widetilde f\left(c^j\right)+\widetilde f(h)\right)=
\frac{1}{k}\sum_{j=0}^{k-1} \widetilde f\left(hc^j\right)=\frac{1}{k}\sum_{j=0}^{k-1} \widetilde f\left(c^jh\right)\\
=\frac{1}{k}\sum_{j=0}^{k-1} \left(\pi\left(c^j\right)\widetilde f\left(h\right)+ \widetilde f \left(c^j\right)\right)=\frac{1}{k}\sum_{j=0}^{k-1} \pi\left(c^j\right)\widetilde f\left(h\right)-v.
\end{multline}
Note that
\begin{multline*}
\frac{1}{k}\sum_{j=0}^{k-1} \pi\left(c^j\right)\widetilde f\left(h\right)=\frac{1}{k}\sum_{j=0}^{k-1} \left( \pi\left(c^j\right)f(h)-\frac{1}{2^m}\sum_{i=0}^{2^m-1} \pi\left(c^{j+i}\right)f(h)\right)\\=\frac{1}{2^m}\sum_{i=0}^{2^m-1}
\left(\frac{1}{k}\sum_{j=0}^{k-1}\pi\left(c^j\right)-\frac{1}{k}\sum_{j=i}^{i + k-1}\pi\left(c^j\right)\right)f(h).
\end{multline*}
Hence,
\begin{equation}\label{eq: bound error P_k}
\left\|\frac{1}{k}\sum_{j=0}^{k-1} \pi\left(c^j\right)\widetilde f\left(h\right)\right\|\le \frac{d_W(h,e_\H)}{2^m}\sum_{i=0}^{2^m-1}\frac{2i}{k}\le \frac{2^m}{k}d_W(h,e_\H).
\end{equation}

Combining~\eqref{eq:get coboundary} and~\eqref{eq: bound error P_k}, we see that $\widetilde f$ is close to a coboundary in the following sense:
\begin{equation}\label{eq:coboundary}
\left\|\widetilde f(h)-\left(\pi(h)v - v\right)\right\|\le \frac{2^m}{k}d_W(h,e_\H).
\end{equation}
If we now write $c^{n^2}=h_1h_2\cdots h_{4n}$ for $h_1,\ldots,h_{4n}\in \left\{a,a^{-1},b,b^{-1}\right\}$, then the cocycle identity for $\widetilde f$ implies the following bound:
\begin{eqnarray}\label{eq:use coboundary}
\nonumber\left\|\widetilde f\left(c^{n^2}\right)\right\|&=&\left\|\sum_{i=0}^{4n-1}\pi(h_1\cdots h_{4n-i-1})\widetilde f(h_{4n-i})\right\|\\\nonumber&\stackrel{\eqref{eq:coboundary}}{\le}& \left\|\sum_{i=0}^{4n-1}\pi(h_1\cdots h_{4n-i-1})(\pi(h_{4n-i})v - v)\right\|+\frac{4n2^m}{k}\\\nonumber&=&
\left\|\pi\left(c^{n^2}\right)v-v\right\|+\frac{4n2^m}{k}\\&\stackrel{\eqref{eq:norm v}}{\lesssim}& \sqrt{k}+\frac{n2^m}{k}.
\end{eqnarray}
The optimal choice for $k$ in~\eqref{eq:use coboundary} is $k\asymp n^{2/3}2^{2m/3}$. For this choice of $k$, \eqref{eq:use coboundary} becomes the desired bound~\eqref{eq:I-P}.
\end{proof}

\section{Proof of Theorem~\ref{thm:main}}\label{sec:nonembed}

As explained in the introduction, using~\cite[Thm. 9.1]{NP08} we may assume without loss of generality that $f\in Z^1(\pi)$ for some action $\pi$ of $\H$ on $X$ by linear isometric automorphisms. We may also assume that $t\ge 8^p$. Let $m$ be the largest integer such that
\begin{equation}\label{eq:def m}
m^m\le \left(\frac{t}{4}\right)^{p/3}.
\end{equation}
Having defined $m$, let $k$ be the smallest integer such that
\begin{equation}\label{eq:def k}
m^{\frac{3}{2p}+\frac{3(k+1)}{p}}\ge t,
\end{equation}
and set
\begin{equation}\label{eq:def l}
\ell\eqdef \left\lceil\frac{6}{p}\log_2m\right\rceil.
\end{equation}
By Lemma~\ref{lem:ab} there exist integers $i\in [k+1,k+m]$ and $j\in [0,2^\ell-1]$ satisfying for all $n\in \N$,
\begin{equation}\label{eq:asymptotic}
\left\|\pi\left(c^{-j2^{i\ell}}\right)P_{(i+1)\ell}
f\left(c^{n^2}\right)-P_{i\ell}f\left(c^{n^2}\right)\right\|\le \frac{16Kn}{m^{1/p}}.
\end{equation}
Choose \begin{equation}\label{eq:def n}
n\eqdef \frac14\left\lceil m^{\frac{3}{2p}}2^{\frac{i\ell}{2}}\right\rceil,
\end{equation}

We may write
\begin{multline*}
f\left(c^{n^2}\right)=\pi\left(c^{-j2^{i\ell}}\right)P_{(i+1)\ell}f\left(c^{n^2}\right) +\left(P_{i\ell}f\left(c^{n^2}\right) -\pi\left(c^{-j2^{i\ell}}\right)P_{(i+1)\ell}f\left(c^{n^2}\right) \right) \\+\left(f\left(c^{n^2}\right)-P_{i\ell}f\left(c^{n^2}\right)\right).
\end{multline*}
Hence, by Lemma~\ref{lem: P_m bound}, inequality~\eqref{eq:asymptotic}, and Lemma~\ref{lem:I-P}, we obtain the following bound:
\begin{equation}\label{eq:before m}
\omega_f(4n)=\omega_f\left(d_W\left(c^{n^2},e_\H\right)\right)\le \left\|f\left(c^{n^2}\right)\right\|
\lesssim \frac{n^{5/3}}{2^{(i+1)\ell/3}}+\frac{8Kn}{m^{1/p}}+2^{i\ell/3}n^{1/3}\\\stackrel{\eqref{eq:def n}\wedge \eqref{eq:def l}}{\lesssim} \frac{Kn}{m^{1/p}}.
\end{equation}
Observe that
\begin{equation*}\label{eq:lower m}
4n\stackrel{\eqref{eq:def n}}{\ge} m^{\frac{3}{2p}}2^{\frac{i\ell}{2}}\ge m^{\frac{3}{2p}}2^{\frac{(k+1)\ell}{2}}
\stackrel{\eqref{eq:def l}}{\ge} m^{\frac{3}{2p}+\frac{3(k+1)}{p}}\stackrel{\eqref{eq:def k}}{\ge} t.
\end{equation*}
At the same time,
\begin{equation*}\label{eq:upper m}
4n\stackrel{\eqref{eq:def n}}{\le} 2m^{\frac{3}{2p}}2^{\frac{i\ell}{2}}\le 2m^{\frac{3}{2p}}2^{\frac{(k+m)\ell}{2}}
\stackrel{\eqref{eq:def l}}{\le} 4m^{\frac{3}{2p}+\frac{3k}{p}}\cdot m^{\frac{3m}{p}}\stackrel{\eqref{eq:def k}}{<}4tm^{\frac{3m}{p}}\stackrel{\eqref{eq:def m}}{\le} t^2.
\end{equation*}
Hence $t\le 4n\le t^2$. The definition of $m$ implies that $m\gtrsim
\frac{p}{3}\frac{\log n}{\log\log n}$, and
therefore~\eqref{eq:before m} becomes:
$$
\frac{\omega_f(4n)}{n}\lesssim K\left(\frac{\log\log n}{\log
n}\right)^{1/p}.
$$
The proof of Theorem~\ref{thm:main} is complete.
\qed

\section{Deduction of Theorem~\ref{thm:distortion} from Theorem~\ref{thm:main}}\label{sec:deduce}

Fix $R\ge 4$ and a function $f:B_R\to X$ satisfying
\begin{equation}\label{eq:assume distortion}
\forall\  x,y\in B_R,\quad d_W(x,y)\le \|f(x)-f(y)\|\le Dd_W(x,y).
\end{equation}
Our goal is to bound $D$ from below. Without loss of generality assume that $f(e)=0$. Define $f^*:\H\to X$ by
\begin{equation}\label{eq:extend}
f^*(x)=\left\{\begin{array}{ll} f(x)& x\in B_{R/2},\\
2\left(1-\frac{d_W(x,e)}{R}\right)f(x) & x\in B_R\setminus B_{R/2},\\
0 & x\in \H\setminus B_R.
\end{array}\right.
\end{equation}
Then $f^*$ is $2D$-Lipschitz and coincides with $f$ on $B_{R/2}$. Let $\mathcal N\subseteq \H$ be a maximal $3R$-separated subset of $\H$. Thus the function $f^{**}:\H\to X$ given by $f^{**}(x)\eqdef \sum_{y\in \mathcal N} f^*(y^{-1}x)$ (only one summand is nonzero for any given $x$) is also $2D$-Lipschitz.

\renewcommand{\U}{\mathscr U}

Fix a free ultrafilter $\U$ on  $\N$. Consider the semi-normed space $Y=\left(\ell_\infty(\H,X),\|\cdot\|_Y\right)$, where
$$
\|\psi\|_Y\eqdef \lim_{M\to \U} \left(\frac{1}{|B_M|}\sum_{z\in B_M} \|\psi(z)\|^p\right)^{1/p}.
$$
Note that since $X$ satisfies~\eqref{eq:p-convexity}, so does $Y$. $Y$ is a semi-normed space rather than a normed space, so we should formally deal below with the quotient  $Y/\{f\in \ell_\infty(\H,X):\ \|f\|_Y=0\}$, but we will ignore this inessential formality in what follows. (Complete details are as in the proof of Theorem 9.1 in~\cite{NP08}. Alternatively one can note that our proof of Theorem~\ref{thm:main} carries over without change to the class of semi-normed spaces.)

Define $F:\H\to Y$ by $F(x)(z)\eqdef f^{**}(zx)-f^{**}(z)$. This is well defined since the metric $d_W$ is left-invariant, and therefore $\|F(x)\|_Y\le 2D d_W(x,e_\H)$ for all $x\in \H$. Moreover, by left-invariance, $F$ is $2D$-Lipschitz. Theorem~\ref{thm:main} therefore implies that there exist $x,y\in \H$ such that $\sqrt{R/4}\le d_W(x,y)\le R/4$ and
\begin{equation}\label{eq:use1.3}
\lim_{M\to \U}\left(\frac{1}{|B_M|}\sum_{z\in B_{M}}\left(\frac{\left\|f^{**}(zx)-f^{**}(zy)\right\|}{d_W(x,y)}\right)^p\right)^{1/p}\lesssim
DK_p(X)\left(\frac{\log\log R}{\log R}\right)^{1/p}.
\end{equation}

Fix an integer $M>2d_W(x,e)+4R$ and write $m=M-2d_W(x,e)-4R$. Since $\mathcal N$ is a maximal $3R$-separated subset of $\H$, we have $B_mx\subseteq \bigcup_{w\in \mathcal M} wB_{3R}$, where  $\mathcal M\eqdef \left\{w\in \mathcal N:\ wB_{3R}\cap B_{m}x\neq \emptyset\right\}$. Hence, since for $r\ge 1$ we have $|B_r|\asymp r^4$, we can bound the cardinality of $\mathcal M$ as follows:
\begin{equation}\label{eq:lower M}
|\mathcal M|\gtrsim \left(\frac{M-2d_W(x,e)-4R}{3R}\right)^4.
\end{equation}

If $w\in \mathcal M$ then there exists $z\in B_m$ and $g\in B_{3R}$ such that $zx=wg$. Hence, for every $h\in B_{R/4}$ we have $$d_W(whx^{-1},e)=d_W(zxg^{-1}hx^{-1},e)\le d_W(z,e)+2d_W(x,e)+d_W(g,e)+d_W(h,e)<M.$$ Thus the sets $\{wB_{R/4}x^{-1}\}_{w\in \mathcal M}$ are disjoint and contained in $B_M$. Moreover, if $w\in \mathcal M$ and $z\in wB_{R/4}x^{-1}$ then $d_W(zx,w)\le R/4$, and hence also $d_W(zy,w)\le d_W(zy,zx)+d_W(zx,w)\le R/2$. By the definition of $f^{**}$, this implies that $f^{**}(zx)=f(w^{-1}zx)$ and $f^{**}(zy)=f(w^{-1}zy)$.
Hence,
\begin{multline}\label{eq:use doubling}
\sum_{z\in B_{M}}\left(\frac{\left\|f^{**}(zx)-f^{**}(zy)\right\|}{d_W(x,y)}\right)^p\ge \sum_{w\in \mathcal M}\sum_{z\in wB_{R/4}x^{-1}}\left(\frac{\left\|f(w^{-1}zx)-(w^{-1}zy)\right\|}{d_W(x,y)}\right)^p\\
\stackrel{\eqref{eq:assume distortion}}{\ge} |\mathcal M|\cdot |B_{R/4}|\stackrel{\eqref{eq:lower M}}{\gtrsim} \left(1-\frac{2d_W(x,e)+4R}{M}\right)^4 |B_M|.
\end{multline}
Theorem~\ref{thm:distortion} now follows from~\eqref{eq:use1.3} and~\eqref{eq:use doubling}.\qed

\section{Embeddings into Hilbert space}\label{sec:hilbert}

In this section, we prove Theorem~\ref{thm:ingtro discrete}. We will deduce it from an inequality on cocycles for the real Heisenberg group. We switch to the real Heisenberg group because its representation theory is simpler. However, this comes at the cost of adding a (straightforward) discretization step to the proof. The upshot is that we obtain as a byproduct a smooth Poincar\'e inequality on $\H(\R)$ of independent interest; see Theorem~\ref{Poincarethm}.

The real Heisenberg group  $\H(\R)$ is defined as the matrix group
$$\H(\R)\eqdef\left\{\left(\begin{array}{cccccc}
1 & u & w \\
0 &  1 & v\\
0 & 0 & 1
\end{array}\right):\  u,v,w\in \R\right\}.$$
The discrete Heisenberg group $\H$ sits inside $\H(\R)$ as the
cocompact discrete subgroup consisting of unipotent matrices with
integer coefficients. We equip the group $\H(\R)$ with the word
metric $d_{S_\R}$ associated with the compact symmetric generating
set $S_{\R}=\{a^u,b^v,c^w; |u|,|v|,|w|\leq 1\}$, where
$$a^u=\left(\begin{array}{cccccc}
1 & u & 0 \\
0 &  1 & 0\\
0 & 0 & 1
\end{array}\right), \; b^v=\left(\begin{array}{cccccc}
1 & 0 & 0 \\
0 &  1 & v\\
0 & 0 & 1
\end{array}\right), \; c^w=\left(\begin{array}{cccccc}
1 & 0 & w \\
0 &  1 & 0\\
0 & 0 & 1
\end{array}\right).$$

Let $\mu$ denote a Haar measure on $\H(\R)$, which coincides with
Lebesgue measure under the natural identification of $\H(\R)$ with
$\R^3$.

 \begin{theorem}\label{cocyclethm}
For  every continuous unitary representation $\pi$ of $\H(\R)$, any
continuous cocycle $\gamma\in Z^1(\pi)$ satisfies the inequality
 \begin{equation}\label{cocycle}\int_1^{\infty}\frac{\left\|\gamma\left(c^t\right)\right\|^2}{t^2}dt
  \lesssim \int_{-1}^1  \left(\|\gamma(a^u)\|^2+\|\gamma(b^u)\|^2\right)du.\end{equation}
 \end{theorem}

Section~\ref{sec:irreducible} is devoted to the proof of Theorem
\ref{cocyclethm}. It is clearly enough to check inequality
(\ref{cocycle}) when the representation $\pi$ is irreducible. The
proof  therefore boils down to a quantitative study of $1$-cocycles
with values in an irreducible representation of $\H(\R)$. In the next three subsections we deduce  Theorem
\ref{thm:ingtro discrete} from Theorem~\ref{cocyclethm} by a succession of reductions. Finally, in the last subsection, we state a smooth Poincaré inequality on the real Heisenberg group, whose proof, being very similar to the discrete one, is explained in a few sentences.

\subsection{Proof of Theorem
\ref{cocyclethm}}\label{sec:irreducible}

By the Stone-von Neumann theorem (see for
example~\cite[Ch.~2]{Fol89}), irreducible representations of
$\H(\R)$ are of two types: those that factor through the center,
and, for every $\lambda\in \R\setminus\{0\}$, the representation
$\pi_{\lambda}$ on $L_2(\R)$ satisfying
\begin{equation}\label{eq:def pi lambda} \forall\, h\in
L_2(\R),\quad \pi_{\lambda}(a^ub^vc^w)(h)(x)\eqdef e^{2\pi i\lambda
v x}h(x+u)e^{2\pi i\lambda w}.
\end{equation}

Note that if a nontrivial irreducible representation $\pi$ factors through the
center then
any
$1$-cocycle $\gamma\in Z^1(\pi)$ must vanish on the center.   Indeed,
$\gamma(c^w)$ is invariant under $\pi(\H)$ for every $w\in \R$, which, since the
representation is supposed to be irreducible and nontrivial,
implies that $\gamma(c^w)=0$. Therefore, in proving
Theorem~\ref{cocyclethm} we may assume that $\pi=\pi_{\lambda}$ for
some $\lambda\neq 0$.



By~\cite[Thm.~7]{Gui72} all $1$-cocycles $\gamma\in
Z^1(\pi_\lambda)$ can be approximated uniformly on compact sets by
coboundaries. Hence, it is enough to consider the case where
$\gamma$ is of the form $\gamma(x)=\pi_{\lambda}(x)h-h$, for some
$h\in L_2(\R)$. We may assume that $\|h\|=1$. By the
definition~\eqref{eq:def pi lambda}, for every $w\in \R$ we have
$\|\gamma(c^w)\|^2=4\sin^2(\pi\lambda w),$ from which we deduce that
\begin{equation}\label{eq:w}
\int_1^{\infty}\frac{\|\gamma(c^w)\|^2}{w^2}dw \lesssim
\int_1^{\infty}\frac{\sin^2(\pi\lambda
w)}{w^2}dw=|\lambda|\int_{|\lambda|}^\infty \frac{\sin^2(\pi
w)}{w^2}dw\lesssim \min\{|\lambda|,1\}.
\end{equation}
Also,
\begin{multline}\label{eq:ab grad}
\int_{-1}^1
\left(\|\gamma(a^u)\|^2+\|\gamma(b^u)\|^2\right)du=\int_{-1}^1
\int_\R\left(|h(x+u)-h(x)|^2+2|h(x)|^2\left(1-\cos(2\pi\lambda
ux)\right)\right)dxdu\\
\asymp\int_{-1}^1 \int_\R|h(x+u)-h(x)|^2dxdu+
\int_\R|h(x)|^2\min\{\lambda^2x^2,1\}dx.
\end{multline}


Since  $\R=(\R\setminus [-|u|,|u|])\cup (u+(\R\setminus
[-|u|,|u|]))\cup (-u+(\R\setminus [-|u|,|u|]))$ for every $u\in \R$,
we can bound $\|h\|^2=1$ from above as follows.
\begin{eqnarray}\label{eq:throw interval}
1&\le& \int_{\R\setminus[-|u|,|u|]}
\left(|h(x+u)|^2+|h(x-u)|^2+|h(x)|^2\right)dx\nonumber\\&\lesssim&\nonumber
\int_\R\left(\left|h(x+u)-h(x)\right|^2+\left|h(x)-h(x-u)\right|^2\right)dx+
\int_{\R\setminus[-|u|,|u|]}|h(x)|^2dx\\&\le&
\int_\R\left(\left|h(x+u)-h(x)\right|^2+\left|h(x)-h(x-u)\right|^2\right)dx\nonumber
\\&&\quad+
\frac{1}{\min\{\lambda^2u^2,1\}}\int_{\R}|h(x)|^2\min\{\lambda^2x^2,1\}dx.
\end{eqnarray}
Write $k=\left\lceil 1/\sqrt{|\lambda|}\right\rceil$. By
applying~\eqref{eq:throw interval} with $u=kv$, and integrating over
$v\in [-1,-1/2]\cup[1/2,1]$, we see that
\begin{eqnarray*}
\int_1^{\infty}\frac{\|\gamma(c^w)\|^2}{w^2}dw
&\stackrel{\eqref{eq:w}}{\lesssim}&
\min\{|\lambda|,1\}\\&\stackrel{\eqref{eq:throw
interval}}{\lesssim}& \min\{|\lambda|,1\}\int_{-1}^1 \int_\R|h(x+k
v)-h(x)|^2dxdv\\&&\quad+\frac{\min\{|\lambda|,1\}}
{\min\left\{\lambda^2\left\lceil 1/\sqrt{|\lambda|}\right\rceil^2,1\right\}}\int_\R|h(x)|^2
\min\{\lambda^2x^2,1\}dx\\
&\lesssim& \min\{|\lambda|,1\}k\sum_{j=1}^k\int_{-1}^1 \int_\R|h(x+j
v)-h(x+(j-1)
v)|^2dxdv\\&&\quad+\int_\R|h(x)|^2\min\{\lambda^2x^2,1\}dx\\
&=& \min\{|\lambda|,1\}k^2 \int_{-1}^1
\int_\R|h(x+u)-h(x)|^2dxdu+\int_\R|h(x)|^2\min\{\lambda^2x^2,1\}dx\\
&\stackrel{\eqref{eq:ab grad}}{\lesssim}& \int_{-1}^1
\left(\|\gamma(a^u)\|^2+\|\gamma(b^u)\|^2\right)du.
\end{eqnarray*}
The proof of Theorem~\ref{cocyclethm} is complete.\qed

\subsection{Reduction to finitely supported functions}
\begin{claim}
Inequality (\ref{eq:poin intro}) is a consequence of the following statement.
For every finitely supported $\phi:\H\to L_2$,  we have
\begin{equation}\label{eq:poin rest}
\sum_{x\in
\H}\sum_{k=1}^{\infty}\frac{\left\|\phi(xc^k)-\phi(x)\right\|_2^2}{k^2}\lesssim
\sum_{x\in \H}
\left(\|\phi(xa)-\phi(x)\|^2_2+\|\phi(xb)-\phi(x)\|^2_2\right).
\end{equation}
\end{claim}

\begin{proof} Fix $R\in \N$ and $f:\H\to L_2$. Note that since (\ref{eq:poin intro}) is not sensitive to adding a constant to the function $f$, we can assume without loss of generality that the average of $f$ over $B_{7R}$ is zero.

Define a cutoff function $\xi:\H\to [0,1]$ by
$$
\xi(x)\eqdef \left\{\begin{array}{ll}1& x\in B_{5R},\\
6-\frac{d_W(x,e)}{R}& x\in B_{6R}\setminus B_{5R},\\
0& x\in \H\setminus B_{6R},
\end{array}\right.
$$
and let $\phi \eqdef \xi f$. Then $\phi$ is supported on $B_{6R}$. Since $\xi$ is $1/R$-Lipschitz and takes values in $[0,1]$, for all  $x\in \H$ and $s\in S$,
\begin{multline}\label{eq:tildef}
\|\phi(x)-\phi(xs)\|_2^2\lesssim |\xi(x)-\xi(xs)|^2\cdot\|f(x)\|_2^2+|\xi(xs)|^2\cdot \|f(x)-f(xs)\|_2^2\\\le
\frac{1}{R^2}\|f(x)\|^2+\|f(x)-f(xs)\|_2^2.
\end{multline}
Note that if $k\in \{1,\ldots, R^2\}$ then $d_W(e,c^k)\le 4R$, and hence for $x\in B_R$ we have $xc^k\in B_{5R}$. Therefore, an application of~\eqref{eq:poin rest} to $\phi$ yields the estimate
\begin{multline}\label{eq:cutoff}
\sum_{x\in
B_R}\sum_{k=1}^{R^2}\frac{\|f(xc^k)-f(x)\|_2^2}{k^2}  \le  \sum_{x\in
\H}\sum_{k=1}^{\infty}\frac{\| \phi(xc^k)-\phi(x)\|_2^2}{k^2}
 \lesssim  \sum_{x\in \H}\max_{s\in S}\|\phi(xs)-\phi(x)\|_2^2\\
 =    \sum_{x\in B_{7R}}\max_{s\in S}\|\phi(xs)-\phi(x)\|_2^2                                                                                                        \stackrel{\eqref{eq:tildef}}{\le} \frac{1}{R^2}\sum_{x\in B_{7R}}\|f(x)\|_2^2+  \sum_{x\in B_{7R}}\max_{s\in S}\|f(xs)-f(x)\|_2^2.
\end{multline}
By~\cite[Thm.~2.2]{Kle10} (a discrete version of the classical Heisenberg local Poincar\'e inequality~\cite{Jer86}),
\begin{equation}\label{eq:use Kleiner}
\frac{1}{R^2}\sum_{x\in B_{7R}}\|f(x)\|_2^2\lesssim \sum_{x\in B_{22R}} \left(\|f(xa)-f(x)\|_2^2+\|f(xb)-f(x)\|_2^2\right),
\end{equation}
where we used the fact that the average of $f$ on $B_{7R}$ vanishes. The desired inequality~\eqref{eq:poin intro} is now a consequence of~\eqref{eq:cutoff} and~\eqref{eq:use Kleiner}.
\end{proof}
\subsection{Reduction to an inequality on the real Heisenberg group}


\begin{claim}\label{claim:discretization}
Inequality  (\ref{eq:poin rest})  is a consequence of the following statement.
For  every continuous and compactly supported function $\psi:\H(\R)\to L_2$, we have
\begin{equation}\label{eq:poin cont}
\int_{
\H(\R)}\int_{1}^{\infty}\frac{\left\|\psi(xc^t)-\psi(x)\right\|_2^2}{t^2}dtd\mu(x)\lesssim
\int_{\H(\R)}\left(\sup_{s\in
S_\R}\|\psi(xs)-\psi(x)\|^2_2\right)d\mu(x).
\end{equation}
\end{claim}

\begin{proof} For $r>0$ let $B_r^\R\subseteq \H(\R)$ denote the ball of radius $r$ with respect to the
metric $d_{S_\R}$. Note that $\H(\R)=\bigcup_{g\in \H} gB_2^\R$. Let
$\sigma:\H(\R)\to [0,1]$ be a continuous nonnegative function, which
equals $1$ on $B_2^\R$ and $0$ outside of $B_3^\R$. Let
$\tilde{\sigma}=\sum_{g\in \H}\sigma_g$, where
$\sigma_g(x)=\sigma(g^{-1}x)$. For all $x\in \H(\R)$ we have $1\leq
\tilde{\sigma}(x)\leq C$ for some  $C\in (0,\infty)$. Writing
$\beta=\sigma/\tilde\sigma$ and  $\beta_g(x)=\beta(g^{-1}x)$, we see
that $\{\beta_g\}_{g\in \H}$ is a continuous partition of unity for
$\H(\R)$ satisfying
\begin{equation}\label{eq:grad}
\sup_{x\in \H(\R)}\sum_{g\in \H}\sup_{s\in
S_{\R}}|\beta_g(xs)-\beta_g(x)|<\infty.
\end{equation}

Throughout the ensuing argument we will use repeatedly the fact that
for every $x\in \H(\R)$ the number of elements $g\in \H$ for which
$\beta_g(x)> 0$ is bounded by a constant independent of $x$, and
that $\sum_{g\in \H}\1_{gB_2}\asymp \1_{\H(\R)}$.

Let $\phi: \H\to L_2$ be a finitely supported function on the
discrete Heisenberg group $\H$. Define a function on the real
Heisenberg group $\H(\R)$ by
$$\psi(x)=\sum_{g\in \H}\phi(g)\beta_g(x).$$
Then $\psi$ is compactly supported. We will eventually
apply~\eqref{eq:poin cont} to $\psi$, but before doing so we will
need some preparatory estimates.

The metric $d_{S_\R}$ restricted to $\H\subseteq \H(\R)$ is
bi-Lipschitz equivalent to $d_W$~\cite[Thm.~8.3.19]{BBI01}. It
follows that for all $g_0\in \H$, if $x\in g_0B_3^\R$ then the sum
$\psi(x)-\phi(g_0)=\sum_{g\in \H}\beta_g(x)(\phi(g)-\phi(g_0))$ is
supported on elements of the form $g=g_0h$, with $d_W(h,e)\leq K$,
for some universal constant $K\in \N$. Thus, using~\eqref{eq:grad}
we see that for all $x\in g_0B_2^\R$,
\begin{equation}\label{eq:precrucialright}
\sup_{s\in S_{\R}}\|\psi(xs)-\psi(x)\|_2^2\lesssim \sum_{h\in B_K} \|\phi(g_0h)-\phi(g_0)\|_2^2.
\end{equation}
Integrating (\ref{eq:precrucialright}) over $g_0B_2^\R$ gives the following inequality:
\begin{equation}\label{eq:crucialright}
\int_{g_0B_2^\R}\left(\sup_{s\in S_{\R}}\|\psi(xs)-\psi(x)\|_2^2\right)d\mu(x)\lesssim
\sum_{h\in B_K} \|\phi(g_0h)-\phi(g_0)\|_2^2.
\end{equation}
By summing~\eqref{eq:crucialright} over $g_0\in \H$ we see that
\begin{equation}\label{eq:left}
\int_{\H(\R)}\left(\sup_{s\in S_{\R}}\|
\psi(xs)-\psi(x)\|_2^2\right)d\mu(x) \lesssim \sum_{g\in \H}
\max_{s\in S}\|\phi(g)-\phi(gs)\|_2^2,
\end{equation}
where we used the bound
\begin{equation}\label{eq:use triangle K}
\sum_{g_0\in \H}\sum_{h\in B_K} \|\phi(g_0h)-\phi(g_0)\|_2^2\lesssim
\sum_{g\in \H} \max_{s\in S}\|\phi(g)-\phi(gs)\|_2^2,
\end{equation}
which follows by writing each $h\in B_K$ as a product of at most $K$
elements of $S$, and using the triangle inequality.

In order to deduce from~\eqref{eq:poin cont} a corresponding bound
on $\phi$, we need to bound $\phi$ in terms of $\psi$. To this end,
note that for all $g_0\in \H$,
$$\psi(g_0)  =  \phi(g_0)+ \sum_{g\in \H} \beta(g^{-1}g_0)(\phi(g)-\phi(g_0)) =\phi(g_0)+
\sum_{h\in B_K}\beta(h^{-1})(\phi(g_0h)-\phi(g_0)).$$
It follows
that for all $g_0,g_1\in \H$ we have
\begin{eqnarray}\label{eq:g0g1}
\left\|\phi(g_0)-\phi(g_1)\right\|_2^2\lesssim
\|\psi(g_0)-\psi(g_1)\|_2^2+ \sum_{h\in
B_K}\left(\|\phi(g_0)-\phi(g_0h)\|_2^2+\|\phi(g_1)-\phi(g_1h)\|_2^2\right).
\end{eqnarray}
If $x,y\in \H(\R)$ satisfy
$\max\{d_{S_\R}(x,g_0),d_{S_\R}(y,g_1)\}\leq 2$, then
using~\eqref{eq:precrucialright} we deduce from~\eqref{eq:g0g1} that
\begin{equation}\label{eq:crucialleft}
\left\|\phi(g_0)-\phi(g_1)\right\|_2^2\lesssim
\|\psi(x)-\psi(y)\|_2^2+ \sum_{h\in
B_K}\left(\|\phi(g_0)-\phi(g_0h)\|_2^2+\|\phi(g_1)-\phi(g_1h)\|_2^2\right).
\end{equation}
Fix $g_0\in \H$, $x\in g_0B_1^\R$, $k\in \N$ and $t\in [k,k+1]$.
Writing $g_1=g_0c^k$ and $y=xc^t$, we have $d_{S_\R}(x,g_0)\le 1$
and
$d_{S_\R}(y,g_1)=d_{S_\R}(c^{-k}g_0^{-1}xc^t,e)=d_{S_\R}(g_0^{-1}xc^{t-k},e)\le
d_{S_\R}(g_0^{-1}x,e)+d_{S_\R}(c^{t-k},e)\le 2$. We may therefore
apply~\eqref{eq:crucialleft} and deduce that for all $g_0\in \H$,
$x\in g_0B_1^\R$ and $k\in \N$,
\begin{equation}\label{eq: to integrate}
\|\phi(g_0c^k)-\phi(g_0)\|_2^2\lesssim \|\psi(xc^t)-\psi(x)\|_2^2+
\sum_{h\in
B_K}\left(\left\|\phi(g_0)-\phi(g_0h)\right\|_2^2+\|\phi(g_0c^k)-\phi(g_0c^kh)\|_2^2\right).
\end{equation}
Integrating~\eqref{eq: to integrate} over $x\in g_0B_1^\R$ and $t\in
[k,k+1]$, we see that
\begin{multline}\label{eq:for summing}
\frac{\|\phi(g_0c^k)-\phi(g_0)\|_2^2}{k^2}\lesssim
\int_{g_0B_1^\R}\int_{k}^{k+1}\frac{
\|\psi(xc^t)-\psi(x)\|_2^2}{t^2}dtd\mu(x)\\+\frac{1}{k^2}\sum_{h\in
B_K}\left(\left\|\phi(g_0)-\phi(g_0h)\right\|_2^2+\|\phi(g_0c^k)-\phi(g_0c^kh)\|_2^2\right).
\end{multline}
Since $g_0B_1^\R$ and $g_0'B_1^\R$ intersect at a set of measure
zero if $g_0\neq g_0'$, by summing~\eqref{eq:for summing} over
$g_0\in \H$ and $k\in \N$, we see that
\begin{eqnarray*}
&&\!\!\!\!\!\!\!\!\!\!\!\!\!\!\!\!\!\!\!\!\!\!\!\!\!\!\sum_{g_0\in
\H}\sum_{k=1}^\infty
\frac{\|\phi(g_0c^k)-\phi(g_0)\|_2^2}{k^2}\\&\lesssim&
\int_{\H(\R)}\int_1^\infty\frac{
\|\psi(xc^t)-\psi(x)\|_2^2}{t^2}dtd\mu(x)+\sum_{g_0\in \H}\sum_{h\in
B_K}\left\|\phi(g_0)-\phi(g_0h)\right\|_2^2\\
&\stackrel{\eqref{eq:use triangle
K}}{\lesssim}&\int_{\H(\R)}\int_1^\infty\frac{
\|\psi(xc^t)-\psi(x)\|_2^2}{t^2}dtd\mu(x)+\sum_{g\in \H} \max_{s\in
S}\|\phi(g)-\phi(gs)\|_2^2\\
&\stackrel{\eqref{eq:poin cont}}{\lesssim}&
\int_{\H(\R)}\left(\sup_{s\in
S_\R}\|\psi(xs)-\psi(x)\|^2_2\right)d\mu(x)+\sum_{g\in \H}
\max_{s\in S}\|\phi(g)-\phi(gs)\|_2^2\\
&\stackrel{\eqref{eq:left}}{\lesssim}& \sum_{g\in \H} \max_{s\in
S}\|\phi(g)-\phi(gs)\|_2^2.\qedhere
\end{eqnarray*}
\end{proof}

\subsection{Reduction to a $1$-cocycle on $\H(\R)$}

\begin{claim}\label{claim:endgame}
Inequality  (\ref{eq:poin cont})  follows from Theorem
\ref{cocyclethm}.
\end{claim}
\begin{proof}
Let $\psi:\H(\R)\to L_2$ be continuous and supported in $B_r^\R$ for
some $r\ge 1$. Take a maximal family
$\left\{x_iB_{10r}^\R\right\}_{i=1}^\infty$ of disjoint balls of
radius $10r$. Define $\f:\H(\R)\to L_2$ by
$\varphi(g)=\sum_{i=1}^\infty \psi(x_i^{-1}g)$. Note that since
$d_{S_\R}(c^{r^2},e)\le 4r$, for each $i\in \N$ we have
\begin{equation}\label{eq: psi to varphi left}
\int_{
\H(\R)}\int_{1}^{\infty}\frac{\left\|\psi(xc^t)-\psi(x)\right\|_2^2}{t^2}dtd\mu(x)=
\int_{x_iB_{5r}^\R}\int_{1}^{r^2}\frac{\left\|\varphi(xc^t)-\varphi(x)\right\|_2^2}{t^2}dtd\mu(x).
\end{equation}
Similarly,
\begin{equation}\label{eq: psi to varphi right}
\int_{\H(\R)}\left(\sup_{s\in
S_{\R}}\|\psi(xs)-\psi(x)\|^2_2\right)d\mu(x)=\int_{x_iB_{2r}^\R}\left(\sup_{s\in
S_{\R}}\|\varphi(xs)-\varphi(x)\|^2_2\right)d\mu(x).
\end{equation}

Let $X$ be the space of all finitely supported complex-valued
functions on $\H(\R)$. We denote by $\pi$ the action of $\H(\R)$ on
$X$ given by $\pi(x)\delta_y\eqdef\delta_{xy}$, where
$\delta_z:\H(\R)\to\C$ denotes the function which equals $1$ at
$z\in \H(\R)$ and equals $0$ elsewhere. Let $\U$ be a free
ultrafilter on $\N$. Define a scalar product $[\cdot,\cdot]$ on $X$
by:
\begin{equation}\label{eq:def scalar}
[\delta_x,\delta_y]\eqdef \lim_{n\to \U}
\frac{1}{\mu(B_n^\R)}\int_{B_n^\R} \langle\f(zx),\f(zy)\rangle
d\mu(z),
\end{equation}
where $\langle\cdot,\cdot\rangle$ denotes the scalar product on
$L_2$.

Since $\{B_n\}_{n=1}^\infty$ is a F\o lner sequence for
$\H(\R)$, a limit along a free ultrafilter of averages over $B_n$
when $n\to \infty$ is an invariant mean on $\H(\R)$. It follows that
the scalar product defined in~\eqref{eq:def scalar} is
$\pi(\H(\R))$-invariant. Thus $\pi$ is a unitary representation of
$\H(\R)$  (formally we should first pass to the completion of the quotient of
$X$ by the subspace consisting of norm-zero elements, but we will
ignore this inessential point in what follows). We note that $\pi$ is also continuous in the strong operator topology. Indeed, since $\psi$ is continuous and compactly supported, $\f$ is uniformly continuous. Thus, writing $\|f\|_X^2\eqdef[f,f]$ for  $f\in X$, we have for every $y\in \H(\R)$,
\begin{multline}\label{eq:continuous pi}
\lim_{x\to e} \|\pi(x)\delta_y-\delta_y\|_X^2=\lim_{x\to e} \lim_{n\to \U}
\frac{1}{\mu(B_n^\R)}\int_{B_n^\R} \|\f(gxy)-\f(gy)\|_2^2d\mu(g)\\\le \lim_{x\to e}\sup_{g\in \H(\R)} \|\f(gxy)-\f(gy)\|_2^2=0,
\end{multline}
implying the strong continuity of $\pi$.

Let
$\gamma:\H(\R)\to X$ be given by $\gamma(x)=\delta_x-\delta_e$. Then
$\pi\in Z^1(\pi)$ and for all $x\in X$,
\begin{equation}\label{eq:norm gamma}
\|\gamma(x)\|^2_X=\lim_{n\to \U}
\frac{1}{\mu(B_n^\R)}\int_{B_n^\R} \|\f(gx)-\f(g)\|_2^2d\mu(g).
\end{equation}
Arguing as in~\eqref{eq:continuous pi}, the uniform continuity of
$\f$ and~\eqref{eq:norm gamma} imply that $\gamma$ is a continuous
$1$-cocycle.

Fix  $n>100r$ large enough so as to ensure that we have $\mu(B_n\setminus
B_{n-100r})\le \mu(B_n)/2$. Define $I=\{i\in \N:\  x_iB_{5r}^\R\subseteq
B_n\}$ and write $\Omega=\bigcup_{i\in I} x_iB_{5r}^\R$ and
$\Omega'= \bigcup_{i\in I} x_iB_{10r}^\R$. By the maximality of
$\left\{x_iB_{10r}^\R\right\}_{i=1}^\infty$ we have
$\Omega'\supseteq B_{n-100r}$. Hence $\mu(\Omega')\ge \mu(B_n)/2$.
Since $\H(\R)$ is doubling, $\mu(\Omega')\lesssim \mu(\Omega)$, and
therefore $|I|\mu(B_{5r}^\R)\ge \mu(\Omega)\gtrsim \mu(B_n)$.
Note that if $g\in
x_iB_{5r}^\R$ for some $i\in \N$ then for every $t\in [1,r^2]$ we
have $gc^t\in x_iB_{10r}$. Hence,
\begin{multline}\label{eq:use I}
\frac{1}{\mu(B_n^\R)}\int_{B_n^\R}\|\f(gc^t)-\f(g)\|_2^2d\mu(g)\ge \frac{1}{\mu(B_n^\R)} \sum_{i\in I} \int_{x_iB_{5r}^\R} \|\f(gc^t)-\f(g)\|_2^2d\mu(g)\\
= \frac{|I|}{\mu(B_n^\R)}\int_{B_{5r}^\R}\|\psi(xc^t)-\psi(x)\|_2^2d\mu(x)\gtrsim \frac{1}{\mu(B_{5r}^\R)}\int_{B_{5r}^\R}\|\psi(xc^t)-\psi(x)\|_2^2d\mu(x).
\end{multline}
It follows from~\eqref{eq:norm gamma} and~\eqref{eq:use I} that
\begin{multline}\label{eq:to estimate invariant mean}
\frac{1}{\mu(B_{5r}^\R)}\int_{B_{5r}^\R}\int_1^{r^2}\frac{\|\psi(xc^t)-\psi(x)\|_2^2}{t^2}d\mu(x)\lesssim \int_1^{r^2} \frac{\|\gamma(c^t)\|_X^2}{t^2}dt\stackrel{\eqref{cocycle}}{\lesssim}\sup_{s\in S_\R} \|\gamma(s)\|_X^2 \\
\stackrel{\eqref{eq:norm gamma}}{=}\sup_{s\in S_\R} \lim_{n\to \U}
\frac{1}{\mu(B_n^\R)}\int_{B_n^\R} \|\f(gs)-\f(g)\|_2^2d\mu(g).
\end{multline}
Let $J\subseteq \N$ denote the set of $i\in \N$ such that $B_n\cap x_iB_{2r}\neq\emptyset$. Then $|J|\mu(B_{2r}^\R)\le \mu(B_{n+4r}^\R)$. It follows that for all $s\in S_\R$,
\begin{multline}\label{eq:second doubling}
\int_{B_n^\R} \|\f(gs)-\f(g)\|_2^2d\mu(g)\le  \sum_{i\in J} \int_{x_iB_{2r}}\|\f(gs)-\f(g)\|_2^2d\mu(g)\\=|J| \int_{B_{2r}} \|\psi(xs)-\psi(x)\|_2^2d\mu(x)\le \frac{\mu(B_{n+4r}^\R)}{\mu(B_{2r}^\R)} \int_{B_{2r}} \|\psi(xs)-\psi(x)\|_2^2d\mu(x).
\end{multline}
Substituting~\eqref{eq:second doubling} into~\eqref{eq:to estimate invariant mean} we conclude that
\begin{eqnarray*}
&&\!\!\!\!\!\!\!\!\!\!\!\!\!\!\!\!\!\!\!\!\!\!\!\!\!\!\!\!\!\!\!\!\!\!\!\!\!\!\!\int_{
\H(\R)}\int_{1}^{\infty}\frac{\left\|\psi(xc^t)-\psi(x)\right\|_2^2}{t^2}dtd\mu(x)\stackrel{\eqref{eq: psi to varphi left}}{=}\int_{B_{5r}^\R}\int_1^{r^2}\frac{\|\psi(xc^t)-\psi(x)\|_2^2}{t^2}d\mu(x)\\&\stackrel{\eqref{eq:to estimate invariant mean}\wedge \eqref{eq:second doubling}}{\lesssim}& \frac{\mu(B_{5r}^\R)}{\mu(B_{2r}^\R)}\sup_{s\in S_\R} \lim_{n\to \U}\frac{\mu(B_{n+4r}^\R)}{\mu(B_n^\R)} \int_{B_{2r}} \|\psi(xs)-\psi(x)\|_2^2d\mu(x)\\&\stackrel{\eqref{eq: psi to varphi right}}{\lesssim} &\int_{\H(\R)}\left(\sup_{s\in
S_{\R}}\|\psi(xs)-\psi(x)\|^2_2\right)d\mu(x)
\end{eqnarray*}
 This completes the proof of Claim~\ref{claim:endgame}, and therefore also the proof of Theorem~\ref{thm:ingtro discrete}.
\end{proof}

\subsection{A smooth Poincar\'e inequality on $\H(\R)$} \label{smoothsection}

 Equip $\H(\R)$ with the left-invariant Riemannian metric given
by $du^2+dv^2+(dw-udv)^2$. In what follows, given a smooth function
$f:\H(\R)\to \R$ we let $\nabla_\H f$ denote its gradient with
respect to this Riemannian structure.

\begin{theorem}\label{Poincarethm}
For every smooth function $f:\H(\R)\to \R$, and all $R>0$,
 \begin{equation*}\label{smoothPoincare}\int_{B_R^\R} \int_1^{R^2}\frac{\left|f(xc^t)-f(x)\right|^2}{t^2}dtd\mu(x)
 \lesssim \int_{B_{CR}^\R}\left|\nabla_\H f(x)\right|^2 d\mu(x),
 \end{equation*}
 where $C>0$ is a universal constant.
 \end{theorem}

The proof of this Poincar\'e inequality can be obtained from Theorem
\ref{cocyclethm} in a way similar, and actually even shorter than
its discrete counterpart. Indeed, the discretization step of
Claim~\ref{claim:discretization} is not needed here. The other
difference lies in the first step, where instead of the discrete
Poincar\'e inequality~\eqref{eq:use Kleiner}, we use the following
smooth version, which is due to \cite{Jer86}. For all $R>0$, and for
all smooth functions $f:\H(\R)\to \R$ whose integral over $B_R^\R$ is
zero,
\begin{equation*}
\frac{1}{R^2}\int_{B_{R}^\R} |f(g)|^2d\mu(g)\lesssim
\int_{B_{cR}^\R}|\nabla_\H f(g)|^2d\mu(g),
\end{equation*}
where $c>0$ is a universal constant.

\bibliographystyle{abbrv}
\bibliography{heis}

\end{document}